\documentclass[12pt]{amsart}
\usepackage{amssymb,latexsym}
\usepackage{amsmath}
\usepackage{enumerate}
\usepackage{hyperref}
\usepackage{thmtools}

\makeatletter
\@namedef{subjclassname@2010}{%
  \textup{2010} Mathematics Subject Classification}
\makeatother

\newtheorem*{rem}{Remark}
\newtheorem*{thm}{Theorem}
\declaretheorem[numberwithin=section]{Theorem}
\declaretheorem[numberwithin=section]{Definition}
\declaretheorem[numberwithin=section]{Lemma}
\declaretheorem[numberwithin=section]{Corollary}

\newcommand{\RR}[0]{\mathbb R}

\newcommand{\ZZ}[0]{\mathbb Z}
\newcommand{\CC}[0]{\mathbb C}

\newcommand{\FF}[0]{\mathbb F}

\newcommand{\TT}[0]{\mathbb T}
\newcommand{\cP}[0]{\mathcal P}

\newcommand{\cc}[0]{\textbf{\textit c}}

\newcommand{\vv}[0]{\textbf{\textit{v}}}
\newcommand{\rr}[0]{\textbf{\textit{r}}}

\newcommand{\one}[0]{\mathbf 1}
\renewcommand{\hat}[1]{\widehat{#1}}
\renewcommand{\bar}[1]{\overline{#1}}
\newcommand{\eps}[0]{\varepsilon}

\renewcommand{\mod}[0]{\text{ mod }}

\newcommand{\lr}[1]{\left(#1\right)}
\renewcommand{\quote}[1]{``#1''}

\begin{document}
\title{Character sums over Bohr sets}
\author[Brandon Hanson]{Brandon Hanson}
\address{University of Toronto\\
M5S 2E4 Toronto, Canada}
\email{bhanson@math.toronto.edu}
\allowdisplaybreaks
\maketitle
\begin{abstract}
We prove character sum estimates for additive Bohr subsets modulo a prime.
These estimates are analogous to classical character sum bounds of
P\'olya-Vinogradov and Burgess. These estimates are applied to obtain results on
recurrence mod $p$ by special elements.
\end{abstract}

\section{Introduction}
Let $p$ be a prime number and let $\FF_p$ be the finite field with $p$ elements.
A non-trivial multiplicative character modulo $p$ is a homomorphism
$\chi:\FF_p^\times\to\CC^\times$ which is non-constant. We may abuse notation
and view $\chi$ as a function on the integers defined by $n\mapsto \chi(n\mod p)$ and $\chi(n)=0$ when $p|n$. 
Given a subset $A\subset \FF_p$, we are interested in the sum
\[S(\chi)=\sum_{a\in A}\chi(a).\] Since $\chi$ takes values on the unit
circle, it is always true that $|S(\chi)|\leq |A|$ and when $A$ is a subgroup of
$\FF_p^\times$ this bound is best possible. However, for the typical set $A$
we expect that $|S(\chi)|$ is about $\sqrt{|A|}$. So, to some extent, the size
of $S(\chi)$ as a measure of multiplicative structure of $A$. For instance, the
number of solutions to $ab=cd$ with all variables in $A$ is given by
$\frac{1}{p-1}\sum_\chi |S(\chi)|^4$.

There are classical estimates for $S(\chi)$ when $A$ is an interval.
The first result in this direction is due independently to P\'olya and
Vinogradov. Before stating it, we recall Vinogradov's asymptotic notation. For sequences $X_n$ and $Y_n$,
we take $X_n\ll Y_n$ to mean that $X_n/Y_n\leq c$ for some constant $c$ provided $n$ is suffictiently
large (we shall also write $X_n=O(Y_n)$ to mean the same thing). For sequences $X_n$ and $Y_n$, we take $X_n=o(Y_n)$ to mean that
$X_n/Y_n\to 0$. The following results should be thought of as $p\to \infty$.

\begin{thm}[P\'olya-Vinogradov] 
Let $\chi$ be a non-trivial multiplicative
character modulo $p$. Then \[\left|\sum_{M\leq n\leq M+N}\chi(n)\right|\ll
\sqrt p\log p.\]
\end{thm} 

This estimate is better than the trivial estimate provided $N\gg \sqrt p\log p$
and is simple to prove. In \cite{P}, Paley proved that the bound is in
fact nearly sharp. One needs to work harder to get non-trivial estimates for
shorter intervals. The best result in this direction is due to Burgess.

\begin{thm}[Burgess]
Let $\chi$ be a non-trivial multiplicative character modulo $p$. Then for any
positive integer $k$ and $\eps>0$ we have \[\left|\sum_{M\leq n\leq
M+N}\chi(n)\right|\ll_{k,\eps} N^{1-1/k}p^{(k+1)/4k^2+\eps}.\]
\end{thm}

This result is better than trivial provided $N\gg p^{1/4+\delta}$ which
can be seen by taking the parameter $k$ to be sufficiently large.
Obtaining estimates for even shorter intervals remains a major open
problem in analytic number theory. The interested reader is
referred to Chapter 12 of \cite{IK}.

It is invariance under small translations that allows one to prove such
theorems. Similar theorems are proved for arithmetic progressions by the same
methods. In this paper we prove analogous theorems for sets exhibiting strong
additive structure, namely additive Bohr sets. 
\section{Statement of Results and Applications}
\subsection{Main Results}

Given a subset $\Gamma\subset \FF_p$ and a parameter $\eps>0$, we define the
Bohr set \[B=B(\Gamma,\eps)=\left\{x\in\FF_p:\left\|\frac{xr}{p}\right\|\leq \eps\text{ for each }r\in \Gamma\right\}.\] Here $\|\cdot\|$ denotes the
distance to the nearest integer. Elements $x\in B(\Gamma,\eps)$ dilate
$\Gamma$ into a short interval, and the additive structure of this interval
carries over to $B$. Bohr sets will be discussed further in Section 2.

In Section 3 we obtain the following analog of the P\'olya-Vinogradov estimate
which is non-trivial for large Bohr sets.

\begin{Theorem}[P\'olya-Vinogradov for Bohr sets]
\label{PVBohr}
Let $B=B(\Gamma,\eps)$ be a Bohr set with $|\Gamma|=d$. Then for any non-trivial
multiplicative character $\chi$ \[\left|\sum_{x\in B}\chi(x)\right|\ll_d \sqrt
p(\log p)^d.\]
\end{Theorem}

This result is comparable to \cite{S} in which a P\'olya-Vingradov estimate is
established for generalized arithmetic progressions of rank $d$. For
non-trivial estimates when the Bohr set is on the order of $\sqrt p$
or smaller, we appeal to Burgess' method. We are able to prove
non-trivial results provided the Bohr set satisfies a certain niceness
condition known as \emph{regularity}, see Definition \ref{RegularDefn}. 

\begin{Theorem}[Burgess for Bohr sets]
\label{BurgessBohr}
Let $B=B(\Gamma,\eps)$ be a regular Bohr set with $|\Gamma|=d$. Let $k\geq 1$ be an integer and
let $\chi$ be non-trivial multiplicative character. When
$|B|\geq \sqrt p$ we have the estimate \[\left|\sum_{x\in
B}\chi(x)\right|\ll_{k,d}|B|\cdot
p^{5d/16k^2+o(1)}\lr{\frac{|B|}{\eps^dp}}^{5/16k}\lr{\frac{p}{|B|}}^{-1/8k}.\] When
$|B|<\sqrt p$ we have the estimate \[\left|\sum_{x\in
B}\chi(x)\right|\ll_{k,d}|B|\cdot
p^{{5d/16k^2}+o(1)}\lr{\frac{|B|}{\eps^dp}}^{5/16k}\lr{\frac{|B|^5}{p^2}}^{-1/8k}.\]
\end{Theorem}

The statement appears complicated, but usually one has $|B|\approx
\eps^dp$ so the middle factor in the estimate is harmless. If the rank $d$ is
bounded, one can take $k$ much larger than $d$ and obtain a non-trivial estimate in the range $|B|\gg p^{{2/5}+\delta}$ for some positive $\delta$. This is
comparable to character sum estimates of M.-C. Chang for generalized arithmetic
progressions of comparable rank prove in \cite{C}. As in her proof, we make use
of sum-product phenomena in $\FF_p$.

\subsection{Applications}

Recall Dirichlet's approximation theorem
states that for real numbers $\alpha_1,\ldots,\alpha_d$ there is an integer $n\leq Q$ so that
$\max_k\{\|n\alpha_k\|\}\leq Q^{-1/d}$. Schmidt proved in \cite{Schmidt} that, at
the cost of weakening the approximation, we can take $n$ to be a perfect square.
Specifically, he proved the following.
\begin{thm}
Given real numbers $\alpha_1,\ldots,\alpha_d$ and $Q$ a positive integer, the is
an integer $1\leq n\leq Q$ and a positive absolute constant $c$ such that
\[\max_{1\leq k\leq d}\{\|n^2\alpha_k\|\}\ll d Q^{-c/d^2}.\]
\end{thm}
This result was also proved by Green and Tao in \cite{GT} and extended to
different systems of polynomials in \cite{LM}. An elementary proof of a slightly
weaker estimate was also given in \cite{CLR}.

When $\Gamma$ is a subset of $\FF_p$ and $\eps>0$ then the elements of
$B(\Gamma,\eps)$ are precisely the elements guaranteed by Dirichlet's
approximation theorem. Here we are replacing approximation in the continuous
torus $\RR/\ZZ$ with approximation in the discrete torus $\FF_p$. We will prove
the following $\FF_p$ analog of Schmidt's theorem.

\begin{Theorem}[Recurrence of $k$'th powers]
\label{KthPowers}Let $\Gamma$ be a set of $d$ integers and let $p$ be a prime.
There is an integer $x\leq p$ for which \[\max_{r\in
\Gamma}\left\{\left\|x^k\frac{r}{p}\right\|\right\}\ll_{d} p^{{-1/2d}}\log
p\cdot k^{1/d}.\]
\end{Theorem}

In a similar fashion, we can prove a result about recurrence of primitive roots. 
\begin{Theorem}[Recurrence of primitive roots]
\label{PrimitiveRoots}Let $\Gamma$ be a set of $d$ integers and let $p$ be a prime.
There is an integer $1<x<p$ which generates $\FF_p^\times$ and such that
\[\max_{r\in \Gamma}\left\{\left\|x\frac{r}{p}\right\|\right\}\ll_{d}
\frac{p^{{1/2d}}\log p}{\phi(p-1)^{1/d}}.\]
\end{Theorem}

The remainder of this article is structured as follows. In the next section we
recall necessary facts from Fourier analysis in $\FF_p$, character sums, Bohr
sets and their properties, and sum-product theory. In Section 3
we give the proof of Theorem \ref{PVBohr} and in Section 4 the proof of
Theorem \ref{BurgessBohr}. In Section 5 we present the applications to
recurrence.

\section*{Acknowledgements}

The author would like to thank John Friedlander for helpful discussion. He would
also like to thank the anonymous referee for helpful comments and suggestions.
Part of this research was performed while the author was visiting the Intitute for Pure and Applied Mathematics (IPAM) which is funded by the National Science
Foundation.

\section{Preliminaries}

In this section we describe necessary results from discrete Fourier analysis,
character sums, the theory of Bohr sets, and sum-product theory in $\FF_p$.

\subsection{Discrete Fourier Analysis}
The results in this section are standard. The interested reader is
referred to Chapter 4 of \cite{TV}. We define 
$e_p(a)=e^{2\pi i a/p}$ which is $p$-periodic as a function on $\ZZ$ and
so well-defined on $\FF_p$. For $f:\FF_p\to \CC$ and $q\geq 1$ we have the $L^q$
norm \[\|f\|_q=\lr{\frac{1}{p}\sum_{x\in \FF_p}|f(x)|^q}^{1/q}.\]
The Fourier transform of a function $f$ at $t\in\FF_p$ is
defined as \[\hat f(t)=\sum_{x\in \FF_p}f(x)e_p(-tx).\]

\begin{Lemma}[Properties of the Fourier Transform]
Let $f,g:\FF_p\to\CC$, then we have
\begin{enumerate}
  \item Fourier inversion: $f(x)=\frac{1}{p}\sum_{t\in\FF_p}\hat
  f(t)e_p(tx)$.
  \item Parseval's identity: $\sum_{x\in G}f(x)\bar{
  g(x)}=\frac{1}{p}\sum_{t\in\FF_p}\hat f(t)\bar{\hat g(t)}$.
  \item Plancherel's identity: $\sum_{x\in
  G}|f(x)|^2=\frac{1}{p}\sum_{t\in\FF_p}|\hat f(t)|^2$.
\end{enumerate}
\end{Lemma}

\subsection{Character sums}
Here we recall well-known facts concerning complete character sums over finite
fields. For details, we refer to Chapter 11 of \cite{IK}. Suppose $\chi$ is
a non-trivial multiplicative character. For $x\in \FF_p$ the Fourier transform of $\chi$ at $x$ is
\[\tau(\chi,-x)=\sum_{y\in \FF_p}\chi(y)e_p(-xy)\] which is known as the Gauss
sum.
By expanding the square modulus, it is not hard to prove the following.

\begin{Lemma} 
For non-zero $x\in \FF_p$ we have \[|\tau(\chi,-x)|=\sqrt p\] and
$\tau(\chi,0)=0$.
\end{Lemma}

In the proof of Theorem \ref{BurgessBohr} we shall need Weil's estimate for
character sums with polynomial arguments. 

\begin{thm}[Weil]
Let $f\in\FF_p[x]$ be a polynomial with $r$ distinct roots over $\bar{\FF_p}$.
Then if $\chi$ has order $l$ and provided $f$ is not an $l$'th power over
$\bar{\FF_p}[x]$ we have
\[\left|\sum_{x\in\FF_p}\chi(f(x))\right|\leq r\sqrt p.\]
\end{thm}

\subsection{Bohr sets}\label{BohrFacts}

The material here can be found in Section 4.4 of \cite{TV}. Suppose
$\Gamma\subset\FF_p$ and $\eps>0$ is a parameter, then the Bohr set
$B(\Gamma,\eps)$ is defined as
\[B(\Gamma,\eps)=\left\{x\in\FF_p:\left\|\frac{xr}{p}\right\|\leq \eps\text{ for
each }r\in \Gamma\right\}.\] Here $\|\cdot\|$ is the distance to the nearest
integer, which in this case will be a rational number with denominator $p$.
There are a few ways to view Bohr sets. If we let $I$ be the integer interval
$[-\eps p,\eps p]\cap\ZZ$ (thought of as a subset of $\FF_p$), then $B(\Gamma,\eps)$ consists of those elements $x\in\FF_p$ such that $x\Gamma=\{xr:r\in\Gamma\}\subset I$. Since $\|\theta\|\approx |e^{2\pi i \theta}-1|$, another way to view $B(\Gamma,\eps)$ is as the set of $x\in\FF_p$
such that $|e_p(xr)- 1|\ll\eps$ for $r\in\Gamma$. In this way, $B(\Gamma,\eps)$
is approximately the kernel of the homomorphism $T:\FF_p\to \TT^d$ given by
$T(x)=(e_p(rx))_{r\in\Gamma}$. Since $\FF_p$ has no non-trivial additive
subgroups, Bohr sets are often used as a close approximation.

We have the following estimates on the size of a Bohr set.
\begin{Lemma}
\label{BohrSize}
Let $\Gamma\subset\FF_p$ with $|\Gamma|=d$ and $\eps>0$. Then
\[|B(\Gamma,\eps)|\geq \eps^dp\] and \[|B(\Gamma,2\eps)|\leq 4^d|B(\Gamma,\eps)|.\]
\end{Lemma}
Since $B(\Gamma,\eps)+B(\Gamma,\eps)\subset B(\Gamma,2\eps)$ by the triangle
inequality, we can immediately deduce the following bound.
\begin{Corollary}
\label{BohrDoubling}
Let $\Gamma\subset\FF_p$ with $|\Gamma|=d$ and $\eps>0$. Then
\[|B(\Gamma,\eps)+B(\Gamma,\eps)|\leq 4^d|B(\Gamma,\eps)|.\]
\end{Corollary}

Given $\Gamma\subset \FF_p$, there are certain values of $\eps$ for which
$|B(\Gamma,\eps+\kappa)|$ varies nicely for small values $\kappa$. More
precisely, we define a regular Bohr set as follows.

\begin{Definition}\label{RegularDefn}
Suppose $\Gamma\subset\FF_p$ is a set of size $d$, we say $\eps$ is a regular
value for $\Gamma$ if whenever $|\kappa|<\frac{1}{100d}$ we have 
\[1-100d|\kappa|\leq\frac{|B(\Gamma,(1+\kappa)\eps)|}{|B(\Gamma,\eps)|}\leq1+100d|\kappa|.\]
We say the Bohr set $B(\Gamma,\eps)$ is regular.
\end{Definition} 
The natural first question to ask is if a given $\Gamma$ has any regular values.
In fact, a result due to Bourgain shows that one can always find a regular value
close to any desired radius.

\begin{Lemma}[Bourgain]
\label{RegularValue}
Let $\Gamma$ be a set of size $d$ and let $\delta\in(0,1)$. There
is an $\eps\in(\delta,2\delta)$ which is regular for $\Gamma$.
\end{Lemma}

The crucial property of regular Bohr sets is that they are almost invariant
under translation by Bohr sets of small radius. This allows us to replace a
character sum over a Bohr set by something \quote{smoother}.

\begin{Corollary}
\label{Translation}
Let $B(\Gamma,\eps)$ be a regular Bohr set with $|\Gamma|=d$. If $\eta\leq
\delta\eps/200 d$ for some $0<\delta<1$ then for any natural number $n\geq
1$ and $y_1,\ldots,y_n\in B(\Gamma,\eta)$ and we have
\[\sum_{x\in\FF_p}|\one_{B(\Gamma,\eps)}(x+y_1+\ldots+y_n)-\one_{B(\Gamma,\eps)}(x)|\leq
n\delta|B(\Gamma,\eps)|.\]
\end{Corollary}

\begin{proof}
By the triangle inequality it suffices to prove the result for $n=1$.
For $y=y_1$, the value of
$|\one_{B(\Gamma,\eps)}(x+y)-\one_{B(\Gamma,\eps)}(x)|$ is $0$ unless exactly 
one of $x$ and $x+y$ lies in $B(\Gamma,\eps)$ in which case there is a
contribution of $1$. However, if the latter happens then $x\in
B(\Gamma,\eps+\eta)\setminus B(\Gamma,\eps-\eta)$.
Owing to the regularity of $B(\Gamma,\eps)$, for any $y \in B(\Gamma,\eta),$
there is a contribution of at most \[\left|B\lr{\Gamma,\eps\lr{1+\frac{\delta}{200d}}}\right|-\left|B\lr{\Gamma,\eps\lr{1-\frac{\delta}{200d}}}\right|\leq
\delta\left|B(\Gamma,\eps)\right|.\]
\end{proof}

\subsection{A sum-product estimate}

In order to execute a Burgess type argument for character sums, we shall need
estimates on what is known as multiplicative energy. For two sets $A,B\subset
\FF_p$ we call \[E_\times(A,B)=|\left\{(a_1,a_2,b_1,b_2)\in A\times
A\times B\times B:a_1b_1=a_2b_2\right\}|\] the multiplicative energy
between $A$ and $B$. We observe that if \[r_\times(x)=|\{(a,b)\in A\times
B:ab=x\}|\] then \[E_\times(A,B)=\sum_{x\in
\FF_p}r_\times(x)^2.\] These quantities appear regularly in additive
combinatorics and are closely related to $|A\cdot B|$. Specifically, we shall
need to bound the multiplicative energy between two Bohr sets. To achieve this, make use of the following estimate from \cite{R}\footnote{Recently, Rudnev's sum-product estimate was
improved in \cite{RNRS}. Turning this bound into an energy estimate give a small
improvement to Theorem \ref{BurgessBohr}. However, sum-product estimates are
still far from optimal and an approach incorporating the structure of Bohr sets
would likely be more effective.}.
The estimate presented here is not explicitly mentioned, but it is proved on the way
to proving Theorem 1 of that article.

\begin{Theorem}[Rudnev]
\label{EnergyEstimate}
Let $A\subset \FF_p$ satisfy $|A|<\sqrt p$. Then \[E_\times(A)\ll
|A||A+A|^{7/4}\log |A|.\]
\end{Theorem}

\section{The P\'olya-Vinogradov Argument}
The P\'olya-Vinogradov argument is an effective way of obtaining good character
sum estimates over sets whose Fourier transform has a small $L^1$ norm. Indeed,
suppose $A\subset \FF_p$, then by Parseval's identity and the Gauss sum estimate we have
\[\left|\sum_{a\in A}\chi(a)\right|=\left|\frac{1}{p}\sum_{x\in\FF_p}\hat{\one_A}(x)\tau(\chi,-x)\right|\leq
\sqrt p\|\hat{\one_A}\|_1.\] One can get a fairly strong
estimate on this $L^1$ norm of Bohr sets. We do so now and establish
Theorem \ref{PVBohr}.

\begin{proof}[Proof of Theorem \ref{PVBohr}]
Write $\Gamma=\{r_1,\ldots, r_d\}$ and
$\rr=(r_1,\ldots,r_d)$. Since $x\in B$ if and only if \[r x\in[-\eps p,\eps
p]=I\] for each $r\in \Gamma$, we have

\begin{align*}
\hat{\one_B}(y)&=\sum_{x\in B}e_p(-yx)\\
&=\sum_{x\in \FF_p}\prod_{k=1}^d\one_I(xr_k)e_p(-yx)\\
&=\frac{1}{p^d}\sum_{x\in\FF_p}\prod_{k=1}^d\sum_{v_k\in\FF_p}\hat{\one_I}(v_k)e_p(v_kr_kx)e_p(-yx)\\
&=\frac{1}{p^d}\sum_{\vv\in\FF_p^d}\hat{\one_{I^d}}(\vv)\sum_{x\in\FF_p}e_p(x(\vv\cdot\rr-y))\\
&=\frac{1}{p^{d-1}}\sum_{\substack{\vv\in\FF_p^d\\\vv\cdot\rr=y}}\hat{\one_{I^d}}(\vv).
\end{align*}
Here we have set $I^d=I\stackrel{\times d}{\times\ldots\times} I$ and
\[\hat{\one_{I^d}}((v_1,\ldots,v_d))=\hat{\one_I}(v_1)\cdots\hat{\one_I}(v_d).\]
Plugging this in, we obtain \[\|\hat{\one_B}\|_1\leq
\frac{1}{p^d}\sum_{y\in\FF_p}\sum_{\substack{\vv\in\FF_p^d\\\vv\cdot\rr=y}}|\hat{\one_{I^d}}(\vv)|
=\frac{1}{p^d}\sum_{\vv\in\FF_p^d}|\hat{\one_{I^d}}(\vv)|=\|\hat{\one_I}\|_1^d.\]

As in the classical proof of the P\'olya-Vinogradov inequality,
\[|\hat{\one_I}(v)|=\left|\sum_{k=-N}^{N}e_p(-kv)\right|=\left|\sum_{k=0}^{2N+1}e_p(-kv)\right|
=\left|\frac{e_p(v(2N+2)-1}{e_p(v)-1}\right|\ll\frac{p}{v}.\]
It follows that $\|\hat{\one_I}\|_1\ll \log p$ and the theorem is proved.
\end{proof}

\begin{rem}
If one takes $\Gamma=\{1\}$ and $\eps=N/p$ for some positive integer $N$ then
$B(\Gamma,\eps)=[-N,N]$, thought of as a subset of $\FF_p$. This recovers the
classical P\'olya-Vinogradov estimate \[\sum_{|n|\leq N}\chi(n)\ll\sqrt p\log
p.\]
\end{rem}
\section{The Burgess Argument}

In this section we prove Theorem \ref{BurgessBohr}. The method is the same as in
the proof of Burgess' estimate for character sums over an interval, which can be
found in Chapter 12 of \cite{IK}. The main difference lies in estimating the multiplicative energy
between two Bohr sets and for this we use the sum-product result quoted in
Section 2. Sum-product estimates were used for the same purpose in \cite{C} with methods taken from \cite{KS}. It is likely that the argument presented here is not efficient.
Indeed, Bohr sets are highly structured and the current sum-product estimates
are expected to be suboptimal. For example, one of the energy estimates proved
in \cite{C} was improved in \cite{K} using the geometry of numbers. We were unnable to adapt that argument to the present
situation. 

First, we establish a general version of Burgess' argument which is an
application of H\"older's inequality and Weil's bound.\begin{Lemma}
\label{BasicBurgess}
Let $A,B,C\subset \FF_p$ and suppose $\chi$ is a non-trivial multiplicative
character. Define \[r(x)=|\{(a,b)\in A\times B:ab=x\}|.\] Then for any
positive integer $k$, we have the estimate
\begin{align*}
\sum_{x\in\FF_p}r(x)\left|\sum_{c\in
C}\chi(x+c)\right|&\leq(|A||B|)^{1-1/k}E_\times(A,A)^{1/4k}E_\times(B,B)^{1/4k}\cdot\\
&\cdot\lr{|C|^{2k}2k\sqrt p+(2k|C|)^kp}^{1/2k}.
\end{align*}
\end{Lemma}
\begin{proof}
Call the left hand side above $S$. Applying by H\"older's
inequality \begin{align*}
|S|&\leq\lr{\sum_{x\in\FF_p}r(x)}^{1-1/k}\lr{\sum_{x\in\FF_p}r(x)^2}^{1/2k}\lr{\sum_{x\in\FF_p}\left|\sum_{c\in
C}\chi(x+c)\right|^{2k}}^{1/2k}\\
&=T_1^{1-1/k}T_2^{1/2k}T_3^{1/2k}.
\end{align*}
Now $T_1$ is precisely $|A||B|$ and $T_2$ is the multiplicative energy
$E_\times(A,B)$. By the Cauchy-Schwarz inequality, we have
\[E_\times(A,B)\leq\sqrt{E_\times(A,A)E_\times(B,B)}.\]
Expanding $T_3$ and using that
$\bar\chi(y)=\chi(y^{p-2})$, we have
\begin{align*}
T_3&=\sum_{c_1,\ldots,c_{2k}\in
C}\sum_x\chi((x-c_1)\cdots(x-c_k)(x-c_{k+1})^{p-2}\cdots(x-c_{2k})^{p-2})\\
&=\sum_{\cc\in
C^{2k}}\sum_x\chi(f_{\cc}(x)).
\end{align*}
Here $f_{\cc}(t)$ is the
polynomial
\[f_{\cc}(t)=(t-c_1)\cdots(t-c_k)(t-c_{k+1})^{p-2}\cdots(t-c_{2k})^{p-2}.\]
By Weil's theorem, $\sum_x\chi(f_{\cc}(x))\leq 2k\sqrt p$ unless $f_{\cc}$ is an
$l$'th power, where $l$ is the order of $\chi$. If any of the roots $c_i$ of
$f_{\cc}$ is distinct, it occurs with multiplicity 1 or $p-2$, both of which are
prime to $l$ since $l$ divides $p-1$. Hence $f_{\cc}$ is an $l$'th power only
provided all of its roots can be grouped into pairs. So, for all but at most
$\frac{(2k)!}{2^k k!}\leq (2k|C|)^k$ vectors $\cc$, we have the estimate
$2k\sqrt p$ for the inner sum. For the remaining $\cc$ we bound the sum
trivially by $p$. Hence \[T_3\leq |C|^{2k}2k\sqrt p+(2k|C|)^kp.\]
\end{proof}

We now prove Theorem \ref{BurgessBohr}.
 
\begin{proof}
Suppose $\Gamma\subset\FF_p$ has size $d$ and $\eps$ is a regular value for
$\Gamma$. We may as well assume that $\Gamma\neq 0$ for otherwise $B=\FF_p$ and
the result is trivial. Write $B=B(\Gamma,\eps)$ and let $\chi$ be a non-trivial
character of $\FF_p^\times$. Then we wish to estimate \[S(\chi)=\sum_{x\in
B}\chi(x).\] We begin by first using Corollary \ref{Translation}. Let $\eta=
p^{-1/k}\eps/(200 d)$ and let $y\in B(\Gamma,\eta)$. For any natural
number $n\leq p^{1/2k}$ we have

\begin{align*}
S(\chi)&=\sum_{x\in\FF_p}\one_B(x)\chi(x)\\
&=\sum_{x\in\FF_p}\one_B(x+ny)\chi(x)+O\lr{n|B|p^{-1/k}}\\
&=\sum_{x\in B}\chi(x-ny)+O\lr{n|B|p^{-1/k}}.
\end{align*}

Averaging this over all values $1\leq n\leq p^{1/2k}$ and over
all values $y\in B'= B(\Gamma,\eta)\setminus\{0\}$ we obtain
\[S(\chi)\ll\frac{1}{p^{1/2k}|B'|}\sum_{x\in B}\sum_{y\in
B'}\sum_{1\leq n\leq p^{1/2k}}\chi(x-ny)+O\lr{|B|p^{-{1/2k}}}.\] It remains to
estimate \[T(\chi)\ll\frac{1}{p^{1/2k}|B'|}\sum_{x\in B}\sum_{y\in
B'}\sum_{1\leq n\leq p^{1/2k}}\chi(x-ny).\] 

We begin by assuming that $|B|<\sqrt p$. Then, applying Lemma \ref{BasicBurgess}
(where $r(x)$ is now the number of ways of writing $x$ as $ab$ with $a\in B$ and $b\in (B')^{-1}$), we have
\begin{align*} |T(\chi)|&\ll
\frac{1}{p^{1/2k}|B'|}\sum_{x\in\FF_p}r(x)\left|\sum_{1\leq n\leq
p^{1/2k}}\chi(x-n)\right|\\
&\leq\frac{(|B||B'|)^{1-1/k}E_\times(B,B)^{1/4k}E_\times(B',B')^{1/4k}}{p^{1/2k}|B'|}\cdot\\
&\qquad\cdot\lr{2kp^{3/2}+(2k)^kp^{3/2}}^{1/2k}\\
&\leq|B|(|B||B'|)^{-{3/4k}}(|B+B||B'+B'|)^{7/16k}(\log
p)^{1/2k}\sqrt k p^{1/4k}\end{align*} after applying
Theorem \ref{EnergyEstimate}. Applying Corollary \ref{BohrDoubling}, we get the
bound \[|T(\chi)|\ll
|B|(|B||B'|)^{-{5/16k}}4^{7d/8k}(\log
p)^{1/2k}\sqrt k p^{1/4k}.\]
Using Lemma \ref{BohrSize}, \[|B'|\geq
\eta^dp=\lr{\frac{\eps}{p^{1/k}200d}}^dp\] so that 
\[|T(\chi)|\ll_{d,k}|B|\cdot
p^{{5d/16k^2}+o(1)}\lr{\frac{|B|}{\eps^dp}}^{5/16k}\lr{\frac{|B|^{5/2}}{p}}^{-1/4k}.\]

Now if $|B|\geq \sqrt p$, first split $B$ into disjoint sets $B_i$ with
$\sqrt p\ll|B_i|<\sqrt p$. Then
\[|T(\chi)|\ll\frac{|B|}{\sqrt
p}\cdot\frac{1}{p^{1/2k}|B'|}\max_i\sum_{x\in B_i}\sum_{y\in
B'}\left|\sum_{1\leq n\leq p^{1/2k}}\chi(x-ny)\right|.\] Proceeding as
before, this time bounding $|B_i|< \sqrt p$ and $|B_i+B_i|\leq |B+B|$, we obtain
\begin{align*}
|T(\chi)|&\ll|B|(\sqrt p|B'|)^{-3/4k}(|B+B||B'+B'|)^{7/16k}(\log
p)^{1/2k}\sqrt k p^{1/4k}\\
&=\lr{\frac{|B|}{\sqrt
p}}^{3/4k}\lr{|B|(|B||B'|)^{-3/4k}(|B+B||B'+B'|)^{7/16k}}\cdot\\
&\qquad\cdot\lr{(\log p)^{1/2k}\sqrt k p^{1/4k}}\\
&\ll_{d,k}|B|\cdot
p^{{5d/16k^2}+o(1)}\lr{\frac{|B|}{\eps^dp}}^{5/16k}\lr{\frac{p}{|B|}}^{-1/8k}.
\end{align*}
\end{proof}

It is worth remarking that the Burgess estimate just proved gives
a genuine improvement over the P\'olya-Vinogradov estimate in some cases. To see
this, we need a Bohr set whose size is $|B|\approx \eps^d p\approx p^\gamma$
with $2/5<\gamma<1/2$. To find such a set, we need only note that the bound
in Lemma \ref{BohrSize} is sharp on average. Averaging over all subsets of
$\FF_p$ of size $d$ we have (where $I$ is the interval $[-\eps p,\eps p]$)
\begin{align*}
\frac{1}{\binom{p}{d}}\sum_{|A|=d}|B(A,\eps)|&=\frac{1}{\binom{p}{d}}\sum_{|A|=d}\sum_{x\in\FF_p}\prod_{a\in
A}\one_I(ax)\\
&=\frac{1}{\binom{p}{d}}\sum_{|A|=d}\sum_{x\in\FF_p^\times}\prod_{a\in
A}\one_{x^{-1}I}(a)+O(1)\\
&=\frac{1}{\binom{p}{d}}\sum_{x\in\FF_p^\times}\sum_{|A|=d}\prod_{a\in
A}\one_{x^{-1}I}(a)+O(1).
\end{align*}The inner sum vanishes unless $A\subset x^{-1}I$ in which case it contributes $\binom{|I|}{d}$. Thus the total sum is roughly
$\binom{|I|}{d}\binom{p}{d}^{-1}p\asymp \eps^dp$. It follows that for the
typical choice of $A$ of size $d$ and appropriate choice of $\eps$, which we can
take to be regular by Lemma \ref{RegularValue}, we find a regular Bohr set with
size in the desired range.

\section{Application to Polynomial Recurrence}

We are now going to prove Theorem \ref{KthPowers} and Theorem
\ref{PrimitiveRoots}. Their proofs will follow the standard method of counting
with characters. First we prove an analog of Schmidt's theorem for
squares. This proof is quite simple and doesn't need character sums, but it will give a good idea of what to aim for when we move to higher powers. 

Let $\Gamma\subset \FF_p$ be a set of size $d$ and let
$\eps>0$ be a parameter. Then $B=B(\Gamma,\eps)$ contains a non-zero square
provided $\eps^{2d}p>1$. To see this, observe that Bohr sets have the dilation
property $xB=B(x^{-1}\Gamma,\eps)$, which follows immediately from the
definition of a Bohr set. If the non-zero elements of $B$ are all
non-squares, then for any non-square element $x$, $xB(\Gamma,\eps)\cap
B(\Gamma,\eps)=\{0\}$. But this intersection contains $B(\Gamma\cup x^{-1}\Gamma,\eps)$ which has size at least $\eps^{2d}p$ by Lemma
\ref{BohrSize} yielding a contradiction. It follows that there is a non-zero
integer $1\leq a<p$ such that \[\max_{r\in \Gamma}\left\{\left\|a^2\frac{r}{p}\right\|\right\}\ll p^{-1/2d}.\]

The above argument does not immediately generalize to higher powers because
there is no dichotomy - an element can be in any of the $k$ cosets of the set of $k$'th
powers. Instead, we will use Theorem \ref{PVBohr} to find higher powers and
primitive roots in Bohr sets.

\begin{proof}[Proof of Theorem \ref{KthPowers}]Write $B$ for $B(\Gamma,\eps)$.
Observe that when $(k,p-1)=l$ then the $k$'th powers are the same as the $l$'th powers. So we suppose $k|(p-1)$ and $K$ is the subgroup of $\FF_p^\times$ consisting of the
$k$'th powers. This group has index $k$. The problem is then showing that
$B(\Gamma,\eps)\cap K$ is non-empty. Let $K^\perp$ be the group of multiplicative characters
which restrict to the trivial character on $K$. This group has size
$|K^\perp|=k$. The Poisson Summation Formula, which can be found in Chapter
4 of \cite{TV}, states that \[\one_K(x)=\frac{1}{k}\sum_{\chi\in K^\perp}\chi(x).\] Thus,
\[|K\cap B|=\frac{1}{k}\sum_{\chi\in
K^\perp}\sum_{b\in
B}\chi(b).\] After extracting the contribution from the trivial
character $\chi_0$ this we have \[\left||K\cap
B|-\frac{|B|}{k}\right|\leq \max_{\chi}|S(\chi)|\]
where $S(\chi)=\sum_{b\in B} \chi(b)$ and the maximum is taken over all
non-trivial characters $\chi\in K^\perp$. Thus if we can show that the maximum
value of $|S(\chi)|$ is at most $\frac{|B|}{k}$ then $B$ must contain an
element of $K$. By Theorem \ref{PVBohr}, $B$ contains a $k$'th powers provided 
$|B|\gg_d kp^{1/2}(\log p)^d$ which is certainly the case when
$\eps^d\gg_d kp^{-1/2}(\log p)^{d}$ in view of Lemma \ref{BohrSize}.
Thus \[\max_{r\in \Gamma}\left\{\left\|x^k\frac{r}{p}\right\|\right\}\ll_{d}
p^{{-1/2d}}\log p\cdot k^{1/d}.\]
\end{proof}

We now turn to primitive roots. 
\begin{proof}[Proof of Theorem \ref{PrimitiveRoots}]
We can also find primitive roots in a Bohr set. Recall that the group
$\FF_p^\times$ is cyclic and a primitive element of $\FF_p$ is a generator of
this group. Denote the primitive roots of $\FF_p$ by $\cP$. The
characteristic function of $\cP$ has a nice expansion in terms of characters,
due to Vinogradov, see excercise 5.14 of \cite{LN}:
\[\one_\cP(x)=\frac{\phi(p-1)}{p-1}\sum_{d|(p-1)}\frac{\mu(d)}{\phi(d)}\sum_{\chi_d}\chi(x)\]
where $\phi$ is Euler's totient function and $\sum_{\chi_d}$ is the sum over all
characters with order exactly $d$. Summing over the elements of a Bohr
set $B$ and extracting the contribution from the trivial character, we obtain
\[\left||B\cap \cP|-|B|\frac{\phi(p-1)}{p-1}\right|\ll_d \sqrt p(\log p)^d.\] We
deduce that $B$ will contain a primitive root whenever $\eps\gg
\frac{p^{1/2d}}{\phi(p-1)^{1/d}}\cdot \log p$. Thus there is a
primitive root $1<x<p$ with \[\max_{r\in \Gamma}\left\{\left\|x\frac{r}{p}\right\|\right\}\ll_{d}
\frac{p^{1/2d}\log p}{\phi(p-1)^{1/d}}.\]
\end{proof}

We close by mentioning that use of Theorem \ref{BurgessBohr} would allow for
smaller choices of $\eps$ but for the factor $\lr{\frac{|B|}{\eps^dp}}^k$
appearing in the estimate. As we mentioned in the preceding section, this
factor is usually harmless, but we wanted uniform results for all
sets $\Gamma$ which comes more easily by way of Theorem \ref{PVBohr}. 


\begin{thebibliography}{HD}

\normalsize
\baselineskip=17pt



\bibitem[BGK]{BGK} J. Bourgain, A. A. Glibichuk and S. V. Konyagin,
\emph{Estimates for the number of sums and products and for exponential sums in
fields of prime order}, J. Lond. Math. Soc. (2) 73 (2006), no. 2, 380-398.

\bibitem[BKT]{BKT} J. Bourgain, N. Katz and T. Tao,
\emph{A sum-product estimate in finite fields, and applications}, Geom. Funct.
Anal. 14 (2004), no. 1, 27-57. 

\bibitem[C]{C} M.-C. Chang,
\emph{On a question of Davenport and Lewis and new character sum bounds in
finite fields},
Duke Math. J. 145 (2008), no. 3, 409-442. 

\bibitem[CLR]{CLR} E. Croot, N. Lyall and A. Rice,
\emph{A purely combinatorial approach to simultaneous polynomial recurrence
modulo 1}, arXiv:1307.0779.

\bibitem[G]{G} M. Z. Garaev,
\emph{An explicit sum-product estimate in Fp}, Int. Math. Res. Not. IMRN 2007,
no. 11, Art. ID rnm035, 11 pp.

\bibitem[GT]{GT} B. Green and T. Tao, 
\emph{New bounds for Szemer\'edi's theorem. II. A new bound for r4(N),
Analytic number theory}, 180-204, Cambridge University Press, Cambridge, 2009.

\bibitem[IK]{IK} H. Iwaniec and E. Kowalski,
\emph{Analytic Number Theory},  
American Mathematical Society Colloquium
Publications. Amer. Math. Soc., Providence, RI, 2004. 

\bibitem[KS]{KS} N. H. Katz and C.-Y. Shen,
\emph{A slight improvement to Garaev's sum product estimate},  
Proc. Amer. Math. Soc. 136 (2008), 2499-2504.

\bibitem[K]{K} S.V. Konyagin,
\emph{Estimates for character sums in finite fields}, (Russian) Mat. Zametki 88
(2010), no. 4, 529--542; translation in Math. Notes 88 (2010), no. 3-4, 503-515

\bibitem[LN]{LN} R. Lidl and H. Neiderreiter, 
\emph{Finite fields}, Encyclopedia of mathematics and its applications.
Cambridge University Press, 1997.

\bibitem[LM]{LM} N. Lyall and A. Magyar, 
\emph{Simultaneous polynomial recurrence}, Bull. Lond. Math. Soc. 43 (2011), no.
4, 765-785.

\bibitem[P]{P} R. E. A. C. Paley, \emph{A theorem on characters}, J. Lond.
Math. Soc. 7 (1932), 28-32.

\bibitem[RNRS]{RNRS} O. Roche-Newton, M. Rudnev and I. Shkredov,
\emph{New sum-product type estimates over finite fields}, arXiv:1408.0542v1.

\bibitem[R]{R} M. Rudnev,
\emph{An improved sum-product inequality in fields of prime order}, Int.
Math. Res. Not. IMRN 2012, no. 16, 3693-3705.

\bibitem[Sch]{Schmidt} W. M. Schmidt, 
\emph{Small fractional parts of polynomials}, CBMS Regional Conference Series in
Math., 32, Amer. Math. Soc., 1977.

\bibitem[Sh]{S} X. Shao,
\emph{On character sums and exponential sums over generalized arithmetic progressions}, Bull. Lond. Math. Soc. (2013) 45 (3):
541-550.

\bibitem[TV]{TV} T. Tao and V. Vu,
\emph{Additive Combinatorics}, Cambridge Studies in Advanced Mathematics, 105.
Cambridge University Press, Cambridge, 2006.


\end{thebibliography}
\end{document}